\documentclass[11pt]{article}

\usepackage[utf8]{inputenc}
\usepackage[T1]{fontenc}
\usepackage[left=3cm,right=3cm,top=3cm,bottom=3cm]{geometry}

\usepackage[colorlinks=false]{hyperref}
\usepackage{amsmath, amsfonts,amsthm,amssymb}
\usepackage{bbm}
\usepackage{graphicx}
\usepackage{epstopdf}
\usepackage[font=footnotesize,labelfont=bf]{caption}
\usepackage{color}
\usepackage{customCommands}
\usepackage[sort&compress, numbers]{natbib}
\usepackage{snapshot}
\usepackage{booktabs}
\usepackage{multicol}
\usepackage[boxed, ruled, lined]{algorithm2e}
\usepackage[usenames,dvipsnames,svgnames]{xcolor}
\usepackage{blindtext}
\usepackage{enumerate}
\usepackage{csquotes}
\usepackage[font=footnotesize]{subcaption}
\usepackage{tikz}
\usepackage{pgfplots}
\usetikzlibrary{arrows}
\usetikzlibrary{positioning}
\usetikzlibrary{calc}
\usetikzlibrary{shapes.geometric}
\usetikzlibrary{shapes.symbols}
\usetikzlibrary{matrix}
\usetikzlibrary{fit}
\usetikzlibrary{backgrounds}

\newtheorem{theorem}{Theorem}
\newtheorem{lemma}[theorem]{Lemma}
\newtheorem{proposition}[theorem]{Proposition}
\newtheorem{corollary}[theorem]{Corollary}
\theoremstyle{remark}
\newtheorem{remark}[theorem]{Remark}
\theoremstyle{example}

\theoremstyle{theorem}
\newtheorem*{problem*}{Problem}

\title{Mathematical Analysis of the 1D Model and Reconstruction Schemes for Magnetic Particle Imaging}

\author{Erb W.\footnote{joint first authors.}, Weinmann A.$^{*}$, Ahlborg M., Brandt C.,\\ Bringout G., Buzug T.M., 
Frikel J., Kaethner C.,\\ Knopp T., M\"arz T., M\"oddel M., Storath M., Weber A.}

\begin{document}
\maketitle

\newlength\figureheight
\newlength\figurewidth
\setlength\figureheight{0.15\textwidth}

\begin{abstract}	
Magnetic particle imaging (MPI) is a promising new in-vivo medical imaging modality in which distributions of super-paramagnetic nanoparticles are 
tracked based on their response in an applied magnetic field.
In this paper we provide a mathematical analysis of the modeled MPI operator in the univariate situation.
We provide a Hilbert space setup, in which the MPI operator is decomposed into simple building blocks and in which these building blocks are analyzed with 
respect to their mathematical properties.
In turn, we obtain an analysis of the MPI forward operator and, in particular, of its ill-posedness properties.
We further get that the singular values of the MPI core operator decrease exponentially.
We complement our analytic results by some numerical studies which, in particular, suggest a rapid decay of the singular values of the MPI operator.
\end{abstract}

\maketitle

\vspace{2mm}\noindent {\em AMS Subject Classification.} {Primary: 94A12, 92C55, 65R32; Secondary: 44A35, 94A08.} \\[-2mm]


\section{Introduction}

Magnetic particle imaging (MPI) is an emerging medical imaging mo\-dality
whose goal is to determine the spatio-temporal distribution of super-paramagnetic nanoparticles by measuring their non-linear magnetization response 
to an applied magnetic field.
This imaging technology was invented by \mbox{B. Gleich} and J. Weizenecker and published first in 2005 \cite{GleichWeizenecker2005}.
The particle distribution in this modality is not measured directly, but indirectly by the voltage induced by the magnetization of the nanoparticles during the imaging process.
Being more precise, the non-linear response of super-paramagnetic nanoparticles to a temporally oscillating magnetic field is used.
The temporal change of the applied magnetic field results in a temporal change of the magnetization of the nanoparticles and this
change of magnetization induces a voltage in a given set of receive coils. In this way, the induced voltage in the receive coils reflects the distribution of the particles in space.

MPI is a promising imaging modality for biological and medical diagnostics
because of the following two features:
(i) MPI offers a high dynamic spatial and temporal resolution
\cite{KnoppBiederer_etal2010};
(ii) in contrast to many other tomographic methods, such as PET or SPECT
(cf. \cite{SPECT,PET}), MPI does not employ any ionizing radiation during the imaging process.
In addition, in MPI only the magnetic tracer material is imaged which avoids artifacts stemming from nearby tissue as in angiographic CT.
Potential applications for MPI are any tracer based diagnostics. Examples are quantitative stem cell imaging \cite{zheng2013quantitative} 
as well blood flow imaging and cancer detection \cite{BuzugKnopp2012}. 
In angiographic applications, for instance, the magnetic nanoparticles are injected into the blood stream in order to obtain a visualization of the blood flow and 
the vessel system by determining the spatio-temporal distribution of the nanoparticles.

As a basic reference on MPI, we refer to the book \cite{BuzugKnopp2012} and,
for early work, to the articles
\cite{GleichWeizenecker2005,GleichWeizeneckerBorgert2008,Weizenecker_etal2009}. The continuous 1D MPI model studied in this article was 
introduced and investigated in \cite{Rahemeretal2009}.
For a further broader overview about different aspects of MPI imaging we mention the papers
\cite{knopp2011prediction,Knopp_etal2011hm,Knopp_etal2010ec,sattel2009single, goodwill2012x,saritas2013magnetic,konkle2011development,ferguson2009optimization}
and the references therein. Among others, these articles provide information about different designs of MPI scanners, reconstruction methods in various MPI imaging setups as well as
questions concerning the design of suitable nanoparticles for MPI.
Challenges in MPI on its way to application in clinical practice are as well discussed in \cite{BuzugKnopp2012}.
Further actual developments and applications of this technology are given in \cite{Knopp2017,Panagiotopoulos2015}.

The MPI imaging process is commonly modeled as a linear operator equation which has to be inverted (in the sense of inverse problems) to reconstruct the 
tracer distribution from the measured voltage. Depending on the context, the corresponding MPI forward operator can be a system matrix describing a discretized imaging process, or
an integral operator in a continuous model. As this linear system is usually ill-posed, regularization methods have to be employed in the reconstruction process. Typical examples 
are Tikhonov regularization \cite{weizenecker2009three,knopp2010weighted} as well as iterative schemes such as Kaczmarz and Landweber iteration, see \cite{BuzugKnopp2012}.   
Variational regularization schemes for Magnetic Particle Imaging using edge-preserving priors are studied in \cite{storath2017edge}.

The present approaches to represent the MPI forward operator can be categorized into measurement-based approaches, e.g., \cite{WeizeneckerBorgertGleich2007,Rahmer_etal2012}
and model-based approaches, e.g., \cite{KnoppSattel_etal2010,Gruettner_etal2013,Lampe_etal2012}. Further, there exist also some hybrid methods in which these two types of approaches are combined
\cite{Gladiss2017,Gruettner_etal2011}.  
In the measurement-based approach, a basis (or, more generally, a dictionary) is considered and, for each basis element, the system response is measured physically. 
Then, the reconstruction consists in the (regularized) solution of the corresponding system of linear equations
with the columns of the system matrix being the measured system responses. 
The measurement process is time consuming in practice, in particular, when reconstructing  with higher resolution.
For the measurement based approach the receive signal of a small sample of the tracer material is acquired at every pixel position in the FOV. Therefore, the tracer has to be shifted with a robot. As calculated in \cite{Gruettner_etal2013} this takes days for 3D volumes. In order to reduce the time consumption, model-based approaches are of particular interest.

Presently, there are basically three related model-based approaches developed over the past years towards a continuous model for the MPI forward operator.
For a general overview, we refer to \cite{Gruettner_etal2013}.
All these approaches use Faraday's law of induction with respect to a volumetric coil,
see \cite{BuzugKnopp2012,KnoppDiss2010} for a derivation.
Combined with the Langevin model for the magnetization response, a description of the received voltage signal is obtained.
The idealized applied magnetic field consists of a combination of a static and a time-dependent field.
In the following, we describe these approaches in more detail.
In \cite{Rahemeretal2009}, Rahmer et al.\ consider a 1D scenario and study the Fourier series of the signal derived from a delta distribution.
They conclude that, in the case of ideal particles, the system function can be written in terms of Chebyshev polynomials.
They derive a reconstruction formula for the 1D signal. 
For the 3D scenario, the system matrices were studied empirically; the measurements suggest a resemblance to tensor product Chebychev polynomials \cite{Rahmer_etal2012}. In the papers \cite{KnoppSattel_etal2010,KnoppBiederer_etal2010}, 1D and 2D system matrices were modeled rather than measured. 
In this case, the time to calculate a system matrix amounts to a few minutes for simple particle models \cite{Gruettner_etal2013}. 
Goodwill and Conolly  \cite{GoodwillConolly2010, GoodwillConolly2011}
and Schomberg \cite{Schomberg2010} follow a geometric approach.
In \cite{GoodwillConolly2010, GoodwillConolly2011}, the received voltage signal is studied in terms of the trajectory of the so-called field free point (FFP).
The field free point is the time-dependent spatial location in which the applied magnetic field vanishes.
The approach of \cite{Schomberg2010} studies the integral of the received voltage signal instead of the voltage signal itself.
Goodwill and Conolly \cite{GoodwillConolly2010} observe that the field free point in the spatial (or x-space) domain
is uniquely determined by the magnetic drive field and vice versa. A central result is that the signal generation process can be described by a convolution operator.
Schomberg \cite{Schomberg2010} obtains a formulation of convolution type in which the kernel describes the magnetization response of the system.
Based on the x-space model, reconstruction formulas for 2D and 3D MPI are derived in \cite{marz2016model}. 
There, also an analysis of the MPI core operator in 2D and 3D is given. 
Further, in 2D and 3D MPI modelling the particular trajectory traversed by the field-free point plays an important role. 
Such curves and their application in MPI are studied in \cite{DeMarchiErbMarchetti2017,ErbKaethnerAhlborgBuzug2016,ErbKaethnerDenckerAhlborg2015,DenckerErb2017,Kaethner2016IEEE,Knoppetal2009}.

\subsection{Contributions} \vspace{-1mm}
In this paper, we analyze the continuous 1D MPI model from a mathematical point of view.
We make the following three contributions:
(i) we provide a mathematical setup for the MPI reconstruction problem in a Hilbert space setting;
(ii) we provide a mathematical analysis of the MPI forward operator and its building blocks;
(iii) we provide some numerical studies complementing our analytic results.

Concerning (i), we start out to decompose the continuous MPI forward operator into simple building blocks for which we 
provide a mathematical description in terms of operators on suitable Hilbert spaces. This formalization in turn 
allows us to access the forward operator mathematically.
Our main contribution is (ii). Based on the just mentioned decomposition, we first analyze the simple building blocks of the MPI operator and apply these 
results to analyzing the forward operator. In particular, we reveal that the continuous MPI problem in 1D is severely ill-posed.
Further, we obtain that the singular values of the MPI core operator decrease exponentially.
Concerning (iii), we complement our analytic findings by numerical experiments related to the singular value decomposition of the involved operators. In particular, our experiments suggest
a fast decay of the singular values for the entire MPI operator.

\subsection{Outline of the paper} \vspace{-1mm}
We conclude this introduction with recalling the continuous 1D MPI model in Section~\ref{sec:MPImodel}.
Then, in Section~\ref{sec:MathDesc}, we derive a mathematical setup on suitable Hilbert spaces
(Section~\ref{sec:MathFramework}) and decompose the imaging operator in its basic building blocks (Section~\ref{sec:DecompositionOps}).
In Section~\ref{sec:analysis}, we analyze the MPI forward operator.
We analyze its building blocks in Section~\ref{sec:analysisBasicOP} and
their composition as well the entire operator in Section~\ref{sec:AnalysisComposed}.
In Section~\ref{sec:numerics} we present our numerical studies.
Finally, we draw conclusions in Section~\ref{sec:Conclusion}.

\subsection{Modeling MPI - the 1D case}\label{sec:MPImodel}

We briefly describe the continuous MPI model in 1D. For details we refer to \cite{Rahemeretal2009}
or \cite{BuzugKnopp2012}.   

The goal of MPI is to recover the particle density $c(x)$ of super-paramagnetic iron oxid nanoparticles from 
a voltage induced in receive coils caused by the time-dependent
change in the particle magnetization. More precisely, in the univariate situation, the 
general continuous model equation reads 
\begin{equation} \label{sec2:eq1}
 u(t) = - \mu_0 \frac{d}{dt} \int_{\mathrm{eFOV}} \sigma(x) c(x) \Mm(H(x,t)) dx.
\end{equation}
Here, $u(t)$ denotes the voltage induced in the receive coils (up to subtraction of the signal obtained by measuring the generated field without particles), $\sigma(x)$ is the coil sensitivity, 
and $\Mm$ denotes the mean magnetic moment depending on an applied magnetic field $H(x,t)$. 

We adapt some simplifications for the one-dimensional case: we assume that 
the entire field of view of the MPI scanner is $\mathrm{eFOV} \subset \Rr$, and thus $x \in \Rr$. We assume that 
the particle distribution $c(x)$ is entirely supported in $\mathrm{eFOV}$. Further, 
the receive coil sensitivity is assumed to be constant, i.e. $\sigma(x) = \sigma_0$.
Then, the 1D model equation \eqref{sec2:eq1} can be written as 
\begin{equation} \label{sec2:eq2}
 u(t) = - \mu_0 \sigma_0\frac{d}{dt} \int_{\Rr} c(x) \Mm(H(x,t)) dx = \int_{\Rr} c(x) s(x,t) dx 
\end{equation}
with the system kernel
\begin{equation} \label{sec1:eq3}   
s(x,t) = - \mu_0 \sigma_0 \frac{d}{dt} \Mm(H(x,t)) = - \mu_0 \sigma_0 \frac{d \Mm}{dH} \frac{d}{dt}H(x,t).
\end{equation}

We further simplify \eqref{sec1:eq3}. To this end, we suppose that
the applied magnetic field $H(x,t)$ is ideally given by
\begin{equation}  \label{sec2:eq4}
H(x,t) = \underbrace{G x}_{H_S(x)} - \underbrace{ \gamma(t)}_{H_D(t)}.
\end{equation}
Here, $H_S(x) = G x $ gives a linear field  with gradient $G>0$, and $H_D(t)$ denotes a 
periodic magnetic field $H_D(t) = \gamma(t)$ with period $T>0$ and amplitude $A > 0$  (which does not depend on the space variable $x$). 
Further, the $T$-periodic function $\gamma(t)$ 
is assumed to be even such that $\gamma: [0,\frac{T}{2}] \to [-A,A]$ is bijective. 
The respective frequency of the field $H_D(t)$ is denoted by $\omega_0 = \frac{2 \pi}{T}$. In MPI, the linear field $H_S$ is called selection field, whereas
$H_D$ is called drive field.  

With these assumptions, the system kernel $s(x,t)$ can finally be written as
\begin{equation}  \label{sec2:eq5}
 s(x,t) = - \mu_0 \sigma_0 \Mm'(Gx - \gamma(t)) \gamma'(t). 
 \end{equation}

\section{Mathematical Description of 1D MPI} \label{sec:MathDesc}

The goal of this section is the  introduction of a mathematical setup
on Hilbert spaces to describe the model equation \eqref{sec2:eq2} as a linear equation on a Hilbert space. Further,
we aim at decomposing the imaging operator in its basic building blocks.

\subsection{Mathematical Framework}\label{sec:MathFramework}

In order to put the modeling equation \eqref{sec2:eq2} of the last section into a mathematical framework, we introduce the Hilbert spaces
\begin{align} \label{sec3:eq1}
 X^{\mathrm{space}} &= L^2(\Rr), \\ X^{\mathrm{time}} &= L^2([0, \textstyle T/2])
\end{align}
to represent (equivalence classes of) functions with arguments describing a spatial variable and a temporal variable, respectively.
Then, we introduce the integral operator
\begin{equation} \label{sec3:eq2}
 S^{\mathrm{t}}: X^{\mathrm{space}} \to X^{\mathrm{time}}, \quad S^{\mathrm{t}} c (t) = \int_{\Rr} c(x) s(x,t) dx.
\end{equation}

Using this notation, the 1D-MPI reconstruction problem given in equation \eqref{sec2:eq2} can be formulated as the following
inverse problem:

\begin{problem*}
Given an element $u\in R(S^{\mathrm{t}})$ in the range of the operator $S^{\mathrm{t}}$, find a
Hilbert space element $c \in X^{\mathrm{space}}$ such that
\begin{equation} \label{sec2:eq3}
S^{\mathrm{t}} c = u.
\end{equation}
\end{problem*}

Usually, the right hand side $u$ is measured in terms of its Fourier coefficients. Since the trajectory $\gamma(t)$ is assumed to be even, the voltage signal $u$ can be regarded as an odd $T$-periodic signal and expressed in terms of its Fourier-sinus coefficients $\hat{u}(n)$, $n \in \Nn$. They are given as
\[ \hat{u}(n) = \frac{2}{T} \int_{0}^{\frac{T}{2}} u(t) \sin (n \omega_0 t) \ dt, \quad n \in \Nn.\]
Introducing the Hilbert space $X^{\mathrm{freq}} = l^2(\Nn)$
and the operator
\[ S^{\mathrm{f}}: X^{\mathrm{space}} \to X^{\mathrm{freq}}, \quad S^{\mathrm{f}} c (n) = \widehat{S^{\mathrm{t}} c} (n),\]
the imaging equation \eqref{sec2:eq3} can be reformulated for given frequencies $\hat{u} \in X^{\mathrm{freq}}$ as
\begin{equation} \label{sec2:eq3b}
S^{\mathrm{f}} c = \hat{u}.
\end{equation}

We now investigate different reformulations and adaptions of the imaging
equations in \eqref{sec2:eq3} and \eqref{sec2:eq3b}.
We first introduce a new space variable $y$ with 
 $-\frac{A}{G} \leq y \leq \frac{A}{G}.$  
The corresponding transformation is given by
\[ y= \gamma_G(t)=\frac{1}{G} \gamma(t), \quad t \in \left[0, \frac{T}{2}\right].\]
In addition to the space variable $y$, we will use $\gamma_G(t)$ to denote the mapping between
$t$ and $y$. In this terminology, the inverse mapping is given by $t = \gamma_G^{-1}(y)$. 
Note that the new variable $y$ corresponds to the field-free point in the language of \cite{GoodwillConolly2010}.

In the following, the interval $[-\frac{A}{G}, \frac{A}{G}]$ will be referred to as field of view (FOV)
generated by the magnetic drive field. Compared to the entire field of view $\mathrm{eFOV}$ which describes the whole imaging region of the scanner, the FOV $[-\frac{A}{G}, \frac{A}{G}]$ is
the region in which signals are generated. We will always assume that the drive field FOV is contained in $\mathrm{eFOV}$, i.e. 
$[-\frac{A}{G}, \frac{A}{G}] \subset \mathrm{eFOV}$.
Substitution to the new variable $y$, leads to the system kernel 
\[ s(x,y) = \Mm_G'( x - y)) \; G \frac{1}{G} \gamma'(\gamma_G^{-1}(y)),\] 
where $\Mm_G(x) := - \mu_0 \sigma_0 \Mm(G x)$.
Additionally, we can rewrite the operator equation \eqref{sec2:eq3} w.r.t.\ the new space variable $y$, on
the interval $[-\frac{A}{G}, \frac{A}{G}],$ as
\begin{equation} \label{eq:201710011200}
 S^{\mathrm{y}} c (y) = \gamma_G'(\gamma_G^{-1}(y)) \int_{\Rr} c(x) \Mm_G'( x - y) dx.
\end{equation}
With the further assumption that $\Mm'$ is an even function, we can interchange $x$ and $y$ in $\Mm_G'$ and get
\begin{align*}
 S^{\mathrm{y}} c (y) &= \gamma_G'(\gamma_G^{-1}(y)) \int_{\Rr} c(x) \Mm_G'( y - x) dx  = \gamma_G'(\gamma_G^{-1}(y)) \; ( c \ast \Mm_G')(y),
\end{align*}
or equivalently for $t \in [0,T/2]$:
\begin{align} \label{eq:imagingtime}
	 S^{\mathrm{t}} c (t) = \gamma_G'(t) \; ( c \ast \Mm_G')\!\! \left( \gamma_G(t) \right).
\end{align}
The first derivation of the imaging equation \eqref{eq:imagingtime} can be found in a similar form in the original paper \cite{Rahemeretal2009} and in \cite{marz2016model} in a general multivariate setup.

Typical instances for the trajectory $\gamma(t)$ are given as follows.
\begin{itemize}
\item[(i)] For the cosine trajectory
\begin{align}\label{eq:ChebTra}
\gamma_1(t) = A \cos \frac{2\pi}{T}t, \quad   \gamma_1'(t) = - A \frac{2\pi}{T} \sin \frac{2\pi}{T} t,
\end{align}
the forward operator $S_1^{\mathrm{t}},$ is given by
\begin{align} \label{eq:Chebtra2}
S_1^{\mathrm{t}} c (t) = - \frac{A}{G} \frac{2\pi}{T} \sin \frac{2\pi}{T} t \; \Big( c \ast \Mm_G' \Big)\!\! \left( \frac{A}{G} \cos \frac{2\pi}{T} t \right).
\end{align} 
\item[(ii)] For the sawtooth trajectory
\begin{align}\label{eq:Tra2}
\gamma_2(t) = \left\{ \begin{array}{ll} A \left( 1 - \frac{4t}{T}\right), & \quad \text{if $0 \leq t \leq T/2$} \\
                                        A \left( - 1 + \frac{4t}{T}\right) , & \quad \text{if $T/2 \leq t \leq T$}
                      \end{array} \right.
\end{align}
the derivative on the interval $[0,T/2]$ is given as $\gamma_2'(t) = - \frac{4 A }{T}$ and the forward operator $S_2^{\mathrm{t}}$ can be computed as
\begin{align} \label{eq:Tra3}
S_2^{\mathrm{t}} c (t) = - \frac{A}{G} \frac{4}{T} \; \Big( c \ast \Mm_G' \Big)\!\! \left( \frac{A}{G} \left( 1 - \frac{4t}{T}\right) \right).
\end{align}
\end{itemize}

\subsection{Decomposition of the model operators into basic operators}
\label{sec:DecompositionOps}

To analyze the imaging operators $S^{\mathrm{t}}$,  $S^{\mathrm{y}}$ and  $S^{\mathrm{f}}$, we decompose them into their elementary parts. We introduce the Hilbert space $X^{\mathrm{fov}} = L^2([-\frac{A}{G}, \frac{A}{G}])$ and define the operators
\begin{align}
 Q^{\mathrm{conv}}: X^{\mathrm{space}} \to X^{\mathrm{space}}, \quad & Q^{\mathrm{conv}}f (x) = (f \ast \Mm_G')(x),
\end{align}
which represents the spatial convolution with the derivative of the Langevin function $\Mm_G,$ the projection
\begin{align}
 Q^{\mathrm{fov}}: X^{\mathrm{space}} \to X^{\mathrm{fov}}, \quad & Q^{\mathrm{fov}} f (y) =  f (y),
\end{align}
which restricts $f$ to the field of view, and the trajectory operator
\begin{align}
 Q^{\mathrm{time}}: X^{\mathrm{fov}} \to X^{\mathrm{time}}, \quad & Q^{\mathrm{time}} f (t) =  \gamma_G'(t) \; f (\gamma_G(t)).
\end{align}
This yields the decomposition
\begin{align}
 S^{\mathrm{t}} = Q^{\mathrm{time}} \circ Q^{\mathrm{fov}} \circ Q^{\mathrm{conv}}.
\end{align}
If we use the frequency information of the signal $u$ as measured data, we obtain a related decomposition by additionally
imposing an operator describing the Fourier transform. To this end, we let
\[Q^{\mathrm{fft}}: X^{\mathrm{time}} \to X^{\mathrm{freq}},
\quad (Q^{\mathrm{fft}}f)(n) = \frac{2}{T} \int_{0}^{\frac{T}{2}} f(t) \sin (n \omega_0 t) dt, \quad n \in \Nn.\]
Then, we have the decomposition
\begin{align} \label{eq:123}
S^{\mathrm{f}} = Q^{\mathrm{fft}} \circ S^{\mathrm{t}}
= Q^{\mathrm{fft}} \circ Q^{\mathrm{time}} \circ Q^{\mathrm{fov}} \circ Q^{\mathrm{conv}}.
\end{align}

For the particular trajectories given by \eqref{eq:ChebTra} and \eqref{eq:Tra2} the
frequency operator $S^{\mathrm{f}}$ can be simplified further. As a preparation for the corresponding
statement for \eqref{eq:ChebTra}, we additionally introduce the Hilbert space
\begin{align}
&X^{\mathrm{cheb}} = \ts L^2\left([-\frac{A}{G},\frac{A}{G}], \ts \sqrt{1-\frac{G^2}{A^2}y^2}\right),
\end{align}
the embedding operator
\begin{align}
&Q^{\mathrm{emb}}: X^{\mathrm{fov}} \to X^{\mathrm{cheb}}, \quad Q^{\mathrm{emb}}f = f,
\end{align}
as well as the Chebyshev transform
\begin{align}
Q^{\mathrm{chebT}}&: X^{\mathrm{cheb}} \to X^{\mathrm{freq}}, \quad \notag \\
Q^{\mathrm{chebT}}& f(n) = - \frac{2}{T}  \int_{-\frac{A}{G}}^{\frac{A}{G}}  \ts f(y) U_{n-1} \left(\frac{G}{A} y\right) \sqrt{1-\frac{G^2}{A^2}y^2}  dy.
\end{align}
Here $U_n(x) = \sin((n+1)x)/\sin(x)$ denotes the Chebyshev polynomials of the second kind of degree $n \in \N$.

\begin{figure}   \centering
	\begin{tikzpicture}[description/.style={fill=white,inner sep=2pt}]
	\matrix (m) [matrix of math nodes, row sep=3em,
	column sep=3.5em, text height=1.5ex, text depth=0.25ex,ampersand replacement=\&]
	{ X^{\mathrm{space}} \&  X^{\mathrm{space}}  \& X^{\mathrm{fov}} \&  X^{\mathrm{time}} \& X^{\mathrm{freq}} \\
		{}                   \& {}                   \& {}               \&  X^{\mathrm{cheb}}  \\ };
	\path[->,font=\small]
	(m-1-1) edge node[above] {$ Q^{\textrm{conv}} $} (m-1-2)
	(m-1-2) edge node[above] {$ Q^{\textrm{fov}} $} (m-1-3)
	(m-1-3) edge node[above] {$ Q_1^{\textrm{time}} $} (m-1-4)
	(m-1-4) edge node[above] {$ Q^{\textrm{fft}} $} (m-1-5)
	(m-1-3) edge node[below left] {$ Q^{\textrm{emb}} $} (m-2-4)
	(m-2-4) edge node[below right] {$ Q^{\textrm{cheb}} $} (m-1-5);
	\end{tikzpicture}
	\caption{Diagram of the decomposition of the operator $S_1^f$ for the imaging process based on the cosine trajectory \eqref{eq:ChebTra}  } \label{fig:DiagramDEcompOp}
\end{figure}
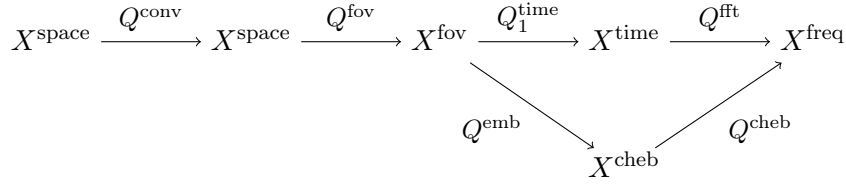

Using this notation, we can give an alternative decomposition of the operator $S_1^{\mathrm{f}}$. For a visualization of the following Lemma \ref{lemma:1},
we refer to Figure~\ref{fig:DiagramDEcompOp}.	
	
\begin{lemma} \label{lemma:1}	
		For the cosine trajectory given by \eqref{eq:ChebTra}, we have the identity
		\begin{align}\label{eq:decompCos}
		S_1^{\mathrm{f}} &= Q^{\mathrm{chebT}} \circ Q^{\mathrm{emb}} \circ Q^{\mathrm{fov}} \circ Q^{\mathrm{conv}}.
		\end{align}
        We have further the identity
        \[ S_1^{\mathrm{f}} c(n)= \frac{2}{T}  \int_{-\frac{A}{G}}^{\frac{A}{G}}  \ts ( c \ast \Mm_G')(y)
        U_{n-1} \left(\frac{G}{A} y\right) \!\! \sqrt{1-\frac{G^2}{A^2}y^2}  dy. \]
\end{lemma}

\begin{proof}
    Using the explicit representation \eqref{eq:Chebtra2} of the operator $S_1^{\mathrm{t}}$, we can calculate the operator $S_1^{\mathrm{f}} =
    Q^{\mathrm{fft}} \circ S_1^{\mathrm{t}}$ as
    \begin{align*} S_1^{\mathrm{f}} c (n) &=
    \frac{2}{T}  \int_{0}^{T/2}  \left(- \frac{A}{G} \omega_0  \sin \omega_0 t \right) \;
    \Big( c \ast \Mm_G' \Big)\!\! \left( \frac{A}{G} \cos \omega_0 t \right) \sin (n \omega_0 t) dt\,.
	\end{align*}
    With the coordinate change $y(t) = \gamma_1(t)/G$, we obtain
	\begin{align*} S_1^{\mathrm{f}} c (n) &= \frac{2}{T}  \int_{-\frac{A}{G}}^{\frac{A}{G}}  ( c \ast \Mm_G')(y) \sin \left(n \arccos \frac{G }{A} y\right) dy,\\
	& = \frac{2}{T}  \int_{-\frac{A}{G}}^{\frac{A}{G}}  \ts ( c \ast \Mm_G')(y) U_{n-1} \left(\frac{G}{A} y\right) \sqrt{1-\frac{G^2}{A^2}y^2}  dy,
	\end{align*}
    where $U_{n-1}$ denotes the Chebyshev polynomial of the second kind of degree $n-1$. By the definition of the Chebyshev transform given above, the
    last integral corresponds precisely with $Q^{\mathrm{chebT}} \circ Q^{\mathrm{emb}} \circ Q^{\mathrm{fov}} \circ Q^{\mathrm{conv}} f$.
\end{proof}

\begin{remark} The Hilbert spaces $X^{\mathrm{time}}$, $X^{\mathrm{cheb}}$ and $X^{\mathrm{freq}}$ are isometric isomorphic.
Therefore, in the case the cosine trajectory $\gamma_1$ is used in the imaging operator, the reconstruction from all three Hilbert spaces
gives equivalent results, see \cite{Gruettner_etal2013,KnoppIWMPI2016a} 
\end{remark}

	
We summarize our results on the decomposition of the MPI forward operators in the following theorem.
\begin{theorem}
	For a general trajectory $\gamma$ we have the decomposition	
	\begin{align}\label{eq:DecompGeneral}
	S^{\mathrm{f}}
	= Q^{\mathrm{fft}} \circ Q^{\mathrm{time}} \circ Q^{\mathrm{fov}} \circ Q^{\mathrm{conv}}.
	\end{align}
	For the cosine trajectory given by \eqref{eq:ChebTra}, we further have
	\begin{align}
	S_1^{\mathrm{f}} c(n) &= Q^{\mathrm{chebT}} \circ Q^{\mathrm{emb}} \circ Q^{\mathrm{fov}} \circ Q^{\mathrm{conv}} c (n) \\
                     &= \frac{2}{T}  \int_{-\frac{A}{G}}^{\frac{A}{G}}  \ts ( c \ast \Mm_G')(y) U_{n-1} \left(\frac{G}{A} y\right) \!\! \sqrt{1-\frac{G^2}{A^2}y^2}  dy\,. \notag
	\end{align}	
	For the sawtooth trajectory given by \eqref{eq:Tra2}, we have the explicit representation 	
	\begin{align}
	S_2^{\mathrm{f}} c (n) &= \frac{2}{T}  \int_{-\frac{A}{G}}^{\frac{A}{G}}
( c \ast \Mm_G')(y) \sin \left(\frac{n T}{4}\left(\frac{A}G - y \right)\right) dy\,.
	\end{align}
\end{theorem}
\begin{proof}
The first assertion was derived in \eqref{eq:123}. The second assertion was proven in Lemma~\ref{lemma:1}. The third statement follows as in
the proof of Lemma \ref{lemma:1} by applying the Fourier transform and the coordinate change $y(t) = \gamma_2(t)/G$ to the explicit
representation of the operator $S_2^{\mathrm{t}}$ given in \eqref{eq:Tra3}.
\end{proof}

\begin{remark}[\em Preproceesing and postprocessing operators]
The operators introduced above are the elementary building blocks for the MPI operator describing the model equation \eqref{sec2:eq2}. 
We point out that there are additional  	
pre- and post processing operators involved in the reconstruction process that go beyond the model equation \eqref{sec2:eq2}.
Depending on the employed MPI scanner there are several possible preprocessing steps involved in the signal generation. 
In some cases, the measured voltage signal $u$ is
not exactly an odd function. This measured distortion of $u$ can be corrected using a symmetrization operator. Further artifacts of the physical scanner
can be adjusted by calculating a transfer function \cite{Bente2015,Croft2012}.
In connection with this we also notice that the frequency information corresponding to the excitation frequencies of the system  information may be severely corrupted or not measured at all. 
This results in an
additional projection operator in the frequency domain.  
Finally, also several postprocessing steps can be modeled and formulated as additional operators. An important example here are
merging operations, in which the reconstruction from several different fields of view are patched together to a single reconstructed image \cite{Ahlborg2015b,KnoppThem2015}. 
Other postprocessing operations involve mild deconvolutions \cite{Croft2012,GoodwillConolly2011,saritas2013magnetic} or DC-offset corrections \cite{Lu2013}.
\end{remark}

\section{Mathematical Analysis of the 1D MPI forward model}  \label{sec:analysis}

We start out to study the operators in the decomposition \eqref{eq:DecompGeneral} in Section~\ref{sec:analysisBasicOP};
in Section~\ref{sec:AnalysisComposed}, we analyse the composition of the building blocks and draw our conclusions for the MPI forward model. For an account on inverse problems and related operator equations
we refer to \cite{Engl_Regularization_of_Inverse_Problems,louis1989inverse,bertero1998introduction,kirsch2011introduction}.

\subsection{Analysis of the basic operators.} \label{sec:analysisBasicOP}

Having a look at \eqref{eq:DecompGeneral}, we anticipate that the main work consists of analyzing the
convolution operator $Q^{\mathrm{conv}}$ and the trajectory operator $Q^{\mathrm{time}}.$
This is done after discussing the simpler operators first.

\paragraph*{The operator $Q^{\mathrm{fov}}.$}
The operator $Q^{\mathrm{fov}}$ is a surjective restriction operator with a clear spectral structure. For later use, we record the following statement.
\begin{lemma}
	The restriction operator $Q^{\mathrm{fov}}: X^{\mathrm{space}} \to X^{\mathrm{fov}}$
	given by $Q^{\mathrm{fov}} f = f|_{[-\frac{A}G, \frac{A}{G}]}$ is an orthogonal projection onto the subspace $X^{\mathrm{fov}}$
	with operator norm equal to $1$.
\end{lemma}
\begin{proof}
	This follows directly from the properties $(Q^{\mathrm{fov}})^2 = Q^{\mathrm{fov}}$ and $Q^{\mathrm{fov}*} = Q^{\mathrm{fov}}$, i.e.
	$Q^{\mathrm{fov}}$ is an idempotent, self-adjoint operator on the Hilbert space $X^{\mathrm{space}}$, and, thus, an orthogonal projection
	onto the range $X^{\mathrm{fov}}$.
\end{proof}

\paragraph*{The operator $Q^{\mathrm{fft}}.$} The Fourier operator $Q^{\mathrm{fft}}$ simply performs the sine transform.
For later use we record the following well-known fact.
\begin{lemma}
	The operator $Q^{\mathrm{fft}}: X^{\mathrm{time}} \to X^{\mathrm{freq}}$ is an isometry.
\end{lemma}
Note that this statement is just the Parseval theorem.

\paragraph*{The convolution operator $Q^{\mathrm{conv}}$.}
The central operator in the MPI imaging process, 
\[
Q^{\mathrm{conv}}: X^{\mathrm{space}} \to X^{\mathrm{space}}, \quad  Q^{\mathrm{conv}}f (x) = (f \ast \Mm_G')(x),\]
is a convolution operator. The magnetization function $\Mm_G$ is, for idealized distributions of magnetic particles, modelled as
\[\Mm_G(x) = a L( \beta G x )\]
with the Langevin function
\[ L(x) =  \left( \coth( x) - \frac{1}{ x} \right) \]
and the constants
\[ a = \mu_0 \sigma_0 m, \quad\text{and} \quad \beta = \mu_0 m /(k_B T).\]
Here $m = \pi D^3 M_s / 6$ is the magnetic moment depending on the radius $D$ of the magnetic particle and the 
saturation magnetization $M_s$. $k_B$ denotes the Boltzmann constant, $T$ the temperature, $\mu_0$ the magnetic 
permeability and $\sigma_0$ the coil sensitivity.

The convolution kernel is then given by
\begin{equation}
\Mm_G'(x) = a \beta G L'(\beta G x)\,,
\end{equation}
with
\[L'(x) = \left\{ \begin{array}{ll} \frac1{x^2} - \frac{1}{\sinh(x)^2}, & x \neq 0\,, \\ \frac13, & x = 0\,. \end{array} \right.\]

In the following we use the concept of Sobolev spaces; an account on these spaces can be found im 
\cite{adams2003sobolev}.

\begin{lemma} \label{lemma-convolution}
The convolution kernel $\Mm_G'$ is analytic and the image $R(Q^{\mathrm{conv}})$ is contained in every Sobolev space
$H^s(\mathbb{R})$, $s \geq 0$. In particular, the inverse problem $Q^{\mathrm{conv}} f = g$ is ill-posed.
\end{lemma}

\begin{proof}
The Langevin function $L$ extended into the complex plane is a holomorphic function on the open stripe $\{x+iy \ : - \pi < y < \pi\}$.
Thus, $L$ and therefore also $L'$ and $\Mm_G'$ are analytic function on the real line.
The Fourier transform of the derivative $L'$ is given by the formula
\begin{equation} \label{eq:Fouriertransform}
\mathcal{F}(L')(\omega) = \sqrt{\frac{\pi}{2}}\left(\omega \coth \frac{\pi}{2} \omega - |\omega|\right).
\end{equation}
This function, and thus also the Fourier transform of $\Mm_G'$ decay exponentially as $\omega \to \infty$.
Therefore, we obtain, for all possible Sobolev norms $\|\cdot\|_s,$ a constants $C,C'$ such that
for all $f \in L^2(\Rr)$ the inequality
\[ \| \Mm_G' \ast f \|_s^2 = \int_{\Rr} (1+\omega^2)^{s} |\mathcal{F}(\Mm_G')(\omega) \hat{f} (\omega)|^2 d \omega \leq C \int_{\Rr} |\hat{f}(\omega)|^2
d\omega =
C' \ \|f\|_2^2 \]
holds. This shows the assertion of the lemma.
\end{proof}

\paragraph*{The trajectory operator $Q^{\mathrm{time}}$.}

The influence of the applied MPI trajectory $\gamma_G$ on the MPI model is described by the operator $Q^{\mathrm{time}}$.
A second equivalent description of this operator is obtained by applying the sine transform $Q^{\mathrm{fft}}$.
Due to the additional multiplication with the derivative $\gamma_{G}'(t)$, the trajectory operator $Q^{\mathrm{time}}$ itself is in some cases an
ill posed operator. This is summarized in the next Lemma. For a more detailed study of 
multiplication operators in the framework of inverse problems, we refer to \cite{FreitagHofmann}.

\begin{lemma} \label{lem:qtime}
Let $\gamma_G(t)$ be a continuously differentiable bijection from the time interval $[0,\frac{T}{2}]$ onto the field of view
$[-\frac{A}{G}, \frac{A}{G}]$. Also, the derivative $\gamma_G'$ vanishes only on a set of measure zero.
Then, the trajectory operator $Q^{\mathrm{time}}: X^{\mathrm{fov}} \to X^{\mathrm{time}}$ given
by $Q^{\mathrm{time}} f (t) = \gamma_G'(t) f(\gamma_G(t))$ is bounded and injective.
The inverse problem $Q^{\mathrm{time}} f = g$ is ill-posed if and only if $\gamma_G'(t)$ has at least one root in the interval $[0,\frac{T}{2}]$.
\end{lemma}

\begin{proof}
For the operator $Q^{\mathrm{time}}$, we have
\[ \|Q^{\mathrm{time}} f\|_2^2 = \int_{0}^{T/2} |\gamma_G'(t)|^2 |f(\gamma_G(t))|^2 dt = \int_{-A/G}^{A/G} |\gamma_G'(\gamma^{-1}(x))|
|f(x)|^2 dx. \]
The operator $Q^{\mathrm{time}}$ is therefore a combination of a coordinate change and a multiplication operator in the Hilbert space $X^{\mathrm{space}}$
with the multiplier $\sqrt{|\gamma_G'(\gamma^{-1}(x))|}$. The operator norm of $Q^{\mathrm{time}}$ corresponds
to the operator norm of the multiplication operator. Since $|\gamma_G'(\gamma^{-1}(x))|$ is continuous on $[-A/G,A/G]$, the multiplication operator
and, thus, also $Q^{\mathrm{time}}$ is bounded. Since, by our assumption, the function
$|\gamma_G'(\gamma^{-1}(x))|$ vanishes only on a set of measure zero the operator $Q^{\mathrm{time}}$ is injective. For the characterization
of the ill-posedness of $Q^{\mathrm{time}}$ we consider the inverse operator
$R^{\mathrm{time}}: X^{\mathrm{time}} \to X^{\mathrm{space}}$ given by $R^{\mathrm{time}} f(x) = 1/\gamma_G'(\gamma^{-1}(x)) f (\gamma_G^{-1}(x))$.
If $\gamma_G'$ has no zero in $[0,T/2]$, then $1/\gamma_G'$ is uniformly bounded on $[0,T/2]$ and the operator $R^{\mathrm{time}}$ is continuous.
On the other hand, if $\gamma_G'$ has a root in $[0,T/2]$, then $1/\gamma_G'$ has a singularity in $[0,T/2]$ and the operator
$R^{\mathrm{time}}$ is unbounded. The unboundedness of $R^{\mathrm{time}}$ is equivalent to the ill-posedness of the linear equation.
\end{proof}

For the two main studied trajectories, the cosine function $\gamma_1$ and the sawtooth function $\gamma_2$,
the operator $Q^{\mathrm{time}}$ behaves as follows:
\begin{corollary}
For the cosine trajectory $\gamma_1$, the operator
\begin{equation}
Q_1^{\mathrm{time}} f(t) = - \frac{A}{G} \frac{2\pi}{T} \sin \frac{2\pi}{T} t \; f \left( \frac{A}{G} \cos \frac{2\pi}{T} t \right)
\end{equation}
gives rise to an ill-posed problem. For the sawtooth trajectory $\gamma_2$ the operator
\begin{equation}
Q_2^{\mathrm{time}} f(t) = - \frac{A}{G} \frac{4}{T} \; f \left( \frac{A}{G} \left( 1 - \frac{4t}{T}\right) \right)
\end{equation}
is an isometry and the problem $Q_2^{\mathrm{time}} f = g$ is well-posed.
\end{corollary}

\subsection{Analysis of the composition of the building blocks and the MPI forward operator}\label{sec:AnalysisComposed}

Next, we analyze the composed operator consisting of the convolution and the restriction operator.
To obtain results for the reconstruction of the particle distribution $c$ from the measured voltage signal, 
we assume in this section that the particle density $c$ is supported in the smaller drive field FOV $[-A/G,A/G]$, i.e. the region in which the voltage signal is generated. Mathematically, this implies
$Q^{\mathrm{fov}} c = c$. Please note that on the Hilbert space $X^{\mathrm{space}}$ neither the operator $Q^{\mathrm{fov}}$ nor the operator $Q^{\mathrm{conv}}$
are compact.

\begin{theorem} \label{thm-main1}
	The composed operator $S^{\mathrm{conv}} = Q^{\mathrm{fov}} \circ Q^{\mathrm{conv}} \circ Q^{\mathrm{fov}}$ is self-adjoint,
	positive and its image $R(S^{\mathrm{conv}})$ is contained in every Sobolev space $H^s[-\frac{A}{G}, \frac{A}{G}]$.
	As such,
	\begin{equation*}
	S^{\mathrm{conv}}: L^2[-\tfrac{A}{G}, \tfrac{A}{G}]  \to  H^s[-\tfrac{A}{G}, \tfrac{A}{G}] \quad \text{is compact and injective,}
	\end{equation*}
	for any $s \in \mathbb R.$
	
\end{theorem}

\begin{proof}
	We note that $S^{\mathrm{conv}}$ is self-adjoint and positive since
	\begin{equation*}
	\langle S^{\mathrm{conv}} f, g \rangle_{L^2[-A/G, A/G]}  =
	\langle Q^{\mathrm{conv}}  ( Q^{\mathrm{fov}} f), Q^{\mathrm{fov}}  g \rangle_{L^2(\mathbb R)},
	\end{equation*}
	and $Q^{\mathrm{conv}}$ is self-adjoint and positive as convolution operator
	with its kernel being a positive real-valued function.	
	Since the convolution kernel of $Q^{\mathrm{conv}}$ is an analytic $C^\infty$ function, its application to
	compactly supported elements of the form $Q^{\mathrm{fov}} f$ results in a $C^\infty$ function defined on the real line. Further,
	by Lemma \ref{lemma-convolution}, $Q^{\mathrm{fov}} f$ is contained in every Sobolev space $H^s(\Rr)$, $s \geq 0$. Therefore,
	by restricting the result $Q^{\mathrm{conv}}  ( Q^{\mathrm{fov}} f)$
	to the compact interval $[-\frac{A}{G}, \frac{A}{G}]$,
	we obtain that
	$S^{\mathrm{conv}} f =	Q^{\mathrm{fov}} \circ Q^{\mathrm{conv}} \circ Q^{\mathrm{fov}} f$ is contained in every
	Sobolev space $H^s([-\frac{A}{G}, \frac{A}{G}])$ of integer order $s$, and therefore, by the inclusion relation in the Sobolev scale,
	in every Sobolev space $H^s([-\frac{A}{G}, \frac{A}{G}]),$ $s \in \Rr$.
	
	Next, we deal with the compactness of $S^{\mathrm{conv}}.$
	Instead of embedding the interval $[-\tfrac{A}{G}, \tfrac{A}{G}]$ which represents the field of view into the real line,
	(which is not suitable for showing compactness)
	we here embed it into a larger properly chosen interval $I$. We choose
	$$I=[- \tfrac{4A}{G}, \tfrac{4A}{G}].$$
	We need the blending function $\chi$ having the properties
	$$
	\chi \in C^\infty, \quad   \chi|_{[-2 A/G, 2 A/G]}=1,  \quad  \chi|_{(\mathbb R \setminus [-3 A/G, 3 A/G])}=0.
	$$
	Note that such functions are typically used in differential geometry;
	they can be constructed similar to the Leray-Friedrichs kernel.
	We consider the kernel $k,$ defined on $I,$ by
	\begin{equation}
	k(x) =  \Mm_G'(x) \cdot \chi(x).
	\end{equation}
	We have that $k$ is supported within $I$, that $k$ belongs to the class $C^\infty,$ and that $k$ and all its derivatives are periodic w.r.t.\ $I.$
	Most important, for all $f$ supported in $I,$
	\begin{equation}\label{eqref:EquOfConvs}
	\ f \ast k \ (x) = f \ast \Mm_G' (x),  \qquad \text{ for all $x \in [-A/G,A/G].$ }
	\end{equation}
	We denote the convolution operator with the kernel $k$ by $K,$
	i.e.,
	\begin{equation}
	K: L^2(I) \to H^s(I), \quad      Kf = k \ast f.
	\end{equation}
	We note that, by the same arguments as used for $S^{\mathrm{conv}}$ above, the range of $K$ is contained in $H^s,$
	and therefore, $K$ is well defined.
	In particular, $K$ applied to an element $f$ which is supported in $[-A/G,A/G],$
	results in a periodic function on $I$ with all its derivatives being periodic as well.
	
	Using the restriction operator
	$$P:  L^2(I) \to L^2[-A/G,A/G],\qquad  Pf = f|_{[-A/G,A/G]}, $$
	we can express our target operator $S^{\mathrm{conv}}$
	by
	\begin{equation} \label{eq:SbyK}
	S^{\mathrm{conv}} = P     \circ   K  \circ  P^\ast.
	\end{equation}
	Together with using the introduced operators, this is a consequence of \eqref{eqref:EquOfConvs}.
	(We note that, by slight abuse of notation, $P$ and its adjoint $P^\ast$ act on $H^s$ instead of $L^2$ as well.)
	If we can show that $K$ is compact, then, since $P$ and $P^\ast$ are bounded operators on both $L^2$ and $H^s,$
	we can conclude the compactness of $S^{\mathrm{conv}}$ from the compactness of $K.$
	
	In the following we show the compactness of $K.$		
	We may restrict to a nonnegative integer $s \in \mathbb{N}_0.$
	If the statement is shown for $s \in \mathbb{N}_0,$
	we may, for general $s' \in \mathbb R,$	
	take any integer $s \geq s',$ and conclude the compactness
	of the operator to $H^{s'}$ from its compactness as an operator to  $H^{s}$ together with the
	continuity of the embedding $H^{s} \hookrightarrow H^{s'}.$
	We consider the Fourier basis $(e_n)$ of complex exponentials $e_n$ with frequency $n$ w.r.t.\ the interval I.
	Since $K$ is a convolution operator, it maps $e_n$ to $\alpha_n e_n,$ with a sequence of scalars $(\alpha_n)$
	which are the Fourier coefficients of the kernel $k$  w.r.t.\ the interval $I.$
	Since $k$ is $C^\infty,$ its Fourier coefficients $(\alpha_n)$ decay faster that $(1+|n|^2))^{s/2},$
	for any $s.$ Further, the norm of the $e_n$ is $O(n^s)$ in $H^s.$
	Therefore, in the coefficient space w.r.t.\ $e_n$ in $L^2$ and $e_n/n^s,$ in $H^s$
	we have that the coefficient sequence $(\beta_n)$ gets mapped to $n^s \alpha_n(\beta_n).$
	Hence, the unit vectors form an eigenbasis and the eigenvalues are isolated except for $0.$
	This implies that $K$ is compact from $L^2(I)$ to $H^s(I),$ for $s \in \mathbb N_0.$
	By the derivations above, we conclude the validity for all $s \in \mathbb R.$
	Then, by \eqref{eq:SbyK}, $S^{\mathrm{conv}}$ is compact. This completes the proof of the compactness. 
    
    For a square integrable function $f$ that is non-zero in the interval $[-A/G,$ $A/G]$, also the convolution
    $Q^{\mathrm{conv}}f$ with a positive kernel is non-zero in $[-A/G,A/G]$. This implies the injectivity of $S^{\mathrm{conv}}$.
\end{proof}

\begin{remark}
	Instead of performing the elementary proof for the compactness of the operator $K$ above,
	it would also have been possible to employ the Rellich embedding theorem for this part.
\end{remark}

The particular composition of the operator $S^{\mathrm{conv}}$ has a remarkable resemblance to related problems treated in the literature, see \cite{SlepianPollak1961,GruenbaumLonghiPerlstadt1982,Widom1964,Erb2013}. In a time-frequency analysis developed by Landau, Pollak and Slepian in the sixties \cite{SlepianPollak1961} a time-frequency operator
of the form $S^{\mathrm{LPS}} = Q^{\mathrm{fov}} \circ Q^{\mathrm{sinc}} \circ P^{\mathrm{fov}}$ was studied. Instead of taking a convolution operator
based on the Langevin function, the operator $Q^{\mathrm{sinc}}$ describes the convolution with the $\mathrm{sinc}$-kernel, i.e. $Q^{\mathrm{sinc}}$ is
a projection onto bandlimited functions inside a frequency interval $I$.
The eigenfunctions of $S^{\mathrm{LPS}}$ are the solutions of a differential equations and
special functions known as prolate spheroidal wave functions. The eigenfunctions
corresponding to the largest eigenvalues turn out to be well localized in the frequency domain inside
the interval $I$. For the actual operator $S^{\mathrm{conv}}$ we expect a similar behavior of the eigenfunctions, i.e. we expect that
the eigenfunctions corresponding to the largest eigenvalues will have a Fourier transform that is to a large part supported at small frequencies.
The asymptotic behavior of the eigenvalues of $S^{\mathrm{conv}}$ can be derived analytically. 
The general result \cite[Theorem II]{Widom1964} on integral equations in combination with the formula \eqref{eq:Fouriertransform} for the Fourier transform of the convolution kernel $\Mm_G'$ yields 
the following decay of the eigenvalues.

\begin{proposition} \label{prop:exponentialdecay}
Let $\sigma_1 \geq \sigma_2 \geq \ldots$ denote the decreasingly ordered eigenvalues of the 
compact positive definite operator $S^{\mathrm{conv}}$. The eigenvalues $\sigma_n$ decay exponentially 
with rate
\[ \ln \sigma_n \sim - n \pi \frac{K(\mathrm{sech} (\beta A))}{K(\mathrm{tanh} (\beta A))},\]
where $K(t) = \int_0^{\pi/2} (1- t^2 \sin^2(\theta))^{-1/2} d \theta$ denotes the
complete elliptic integral of the first kind. 
\end{proposition}

\begin{proof}
By formula \eqref{eq:Fouriertransform}, the Fourier transform of the derivative $L'$ has an asymptotic
exponential decay with rate $\ln (\widehat{L'}(\omega)) \sim - \pi \omega$. This implies that the Fourier transform of $\Mm_G'$ has an asymptotic exponential decay rate of $\ln (\widehat{\Mm_G'}(\omega)) \sim - \frac{\pi}{\beta G} \omega$. Normalizing the field of view $[-A/G, A/G]$ to the 
interval $[-1,1]$ changes the asymptotic exponential decay rate to $-\frac{\pi}{\beta A} \omega$. Now applying \cite[Theorem II]{Widom1964} yields directly the statement of the 
theorem. 
\end{proof}

For the principal imaging operator $S^{\mathrm{t}}$, we obtain the following main statement.

\begin{theorem} \label{thm-main2}
    If $\gamma \in C^{\infty}[0,T/2]$ then the image $R(S^{\mathrm{t}})$ is contained in every Sobolev space $H^s[0,T/2]$, $s \geq 0$.
    Restricting $S^{\mathrm{t}}$ to the space $X^{\mathrm{fov}}$ and assuming that the trajectory $\gamma$ contains only finitely many
    zeros in $[0,T/2]$ the operator
	\begin{equation*}
	S^{\mathrm{t}}: X^{\mathrm{fov}}  \to  H^s[0,T/2] \quad \text{is compact and injective.}
	\end{equation*}
    In particular, the inverse problem stated in \eqref{sec2:eq3} is severely ill-posed.
\end{theorem}

\begin{proof}
From Lemma \ref{lemma-convolution} we know that $ Q^{\mathrm{conv}} c$ is 
contained in every Sobolev space $H^s(\Rr)$. Therefore, since $\gamma$ is infinitely smooth, also the function
$S^{\mathrm{t}} c (t) = \gamma_G'(t) \cdot$ $(Q^{\mathrm{conv}} c)( \gamma_G(t))$ is infinitely times differentiable
and, thus, contained in every Sobolev space $H^s[0,T/2]$. Restricted to the space $X^{\mathrm{fov}}$ 
the operator $S^{\mathrm{t}}$ can be written as
\[ S^{\mathrm{t}} = Q^{\mathrm{time}} \circ S^{\mathrm{\conv}}.\]
Since $S^{\mathrm{\conv}}$ is compact and injective by Theorem \ref{thm-main1} and $Q^{\mathrm{time}}$ is bounded and injective by Lemma \ref{lem:qtime},
the composition $S^{\mathrm{t}}$ is compact and injective. 
\end{proof}

Finally, using the cosine trajectory $\gamma_1$, we get the following additional statement for 
the composed operator $S_1^{\mathrm{freq}}$. By $c_\infty$, we 
denote the space of rapidly decaying sequences, i.e. the space of sequences $(c_n)_{n \in \Nn}$ such 
that $\max_{n} c_n n^{\alpha} < \infty$ for all $\alpha \geq 0$.  

\begin{theorem}
The image $R(S_1^{\mathrm{f}})$ of the operator $S_1^{\mathrm{f}}$ is contained in the space $c_\infty$ 
of rapidly decaying sequences. The operator 
\begin{equation*}
S_1^{\mathrm{f}}: X^{\mathrm{fov}}  \to  c_\infty \quad \text{is compact and injective.}
\end{equation*}
Using the trajectory $\gamma_1$, the inverse problem stated in \eqref{sec2:eq3b} is therefore severely ill-posed. 
\end{theorem}

\begin{proof}
By Lemma \ref{lemma-convolution}, the convolution $ Q^{\mathrm{conv}} c$ is infinitely smooth. The scaled cosine function
$\gamma_{1,G} = A/G \cos \frac{2\pi}{T}t$ is an infinitely smooth $T$-periodic function on the real line. Thus, also
$S_1^{\mathrm{t}} c (t) = \gamma_{1,G}'(t) (Q^{\mathrm{conv}} c) ( \gamma_{1,G}(t))$ is an infinitely times differentiable
and $T$-periodic function on the real line. Additionally, since $\gamma_{1,G}$ is even and $\gamma_{1,G}'$ is odd, the function $S_1^{\mathrm{t}} c$
is an odd function in $t$. Therefore, we can conclude that $S_1^{\mathrm{t}} c$ is an odd $T$-periodic Schwartz function. This implies
that $S_1^{\mathrm{t}} c$ can be expanded in a sine series in which the sequence of coefficients $S_1^{\mathrm{f}} c(n)$, $n \in \Nn$, decays rapidly. 

The operator $S_1^{\mathrm{f}}$ restricted to the space $X^{\mathrm{fov}}$ can be written as 
$S_1^{\mathrm{f}} = Q^{\mathrm{ft}} \circ S_1^{\mathrm{t}} \circ Q^{\mathrm{fov}}$. The composition $S_1^{\mathrm{t}} \circ Q^{\mathrm{fov}}$ is 
injective and compact by Theorem \ref{thm-main2} and $Q^{\mathrm{ft}}$ is an isometry. Thus, also $S_1^{\mathrm{f}}$ is compact and injective
if restricted to $X^{\mathrm{fov}}$. 
\end{proof}

\section{Numerical Studies}  \label{sec:numerics}

\begin{figure} \centering
	\includegraphics[width=0.45\textwidth]{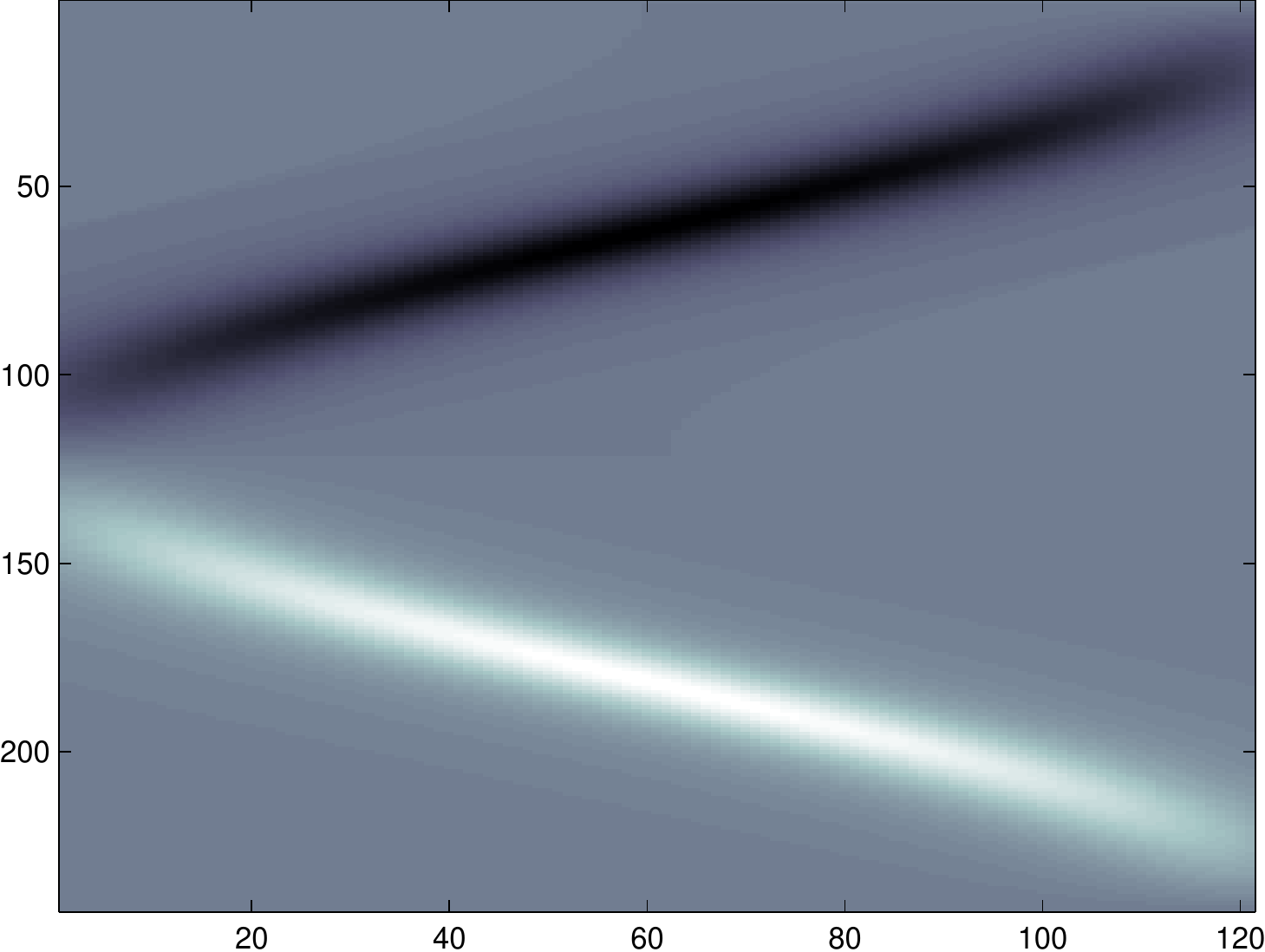} \hspace{4mm}
	\includegraphics[width=0.45\textwidth]{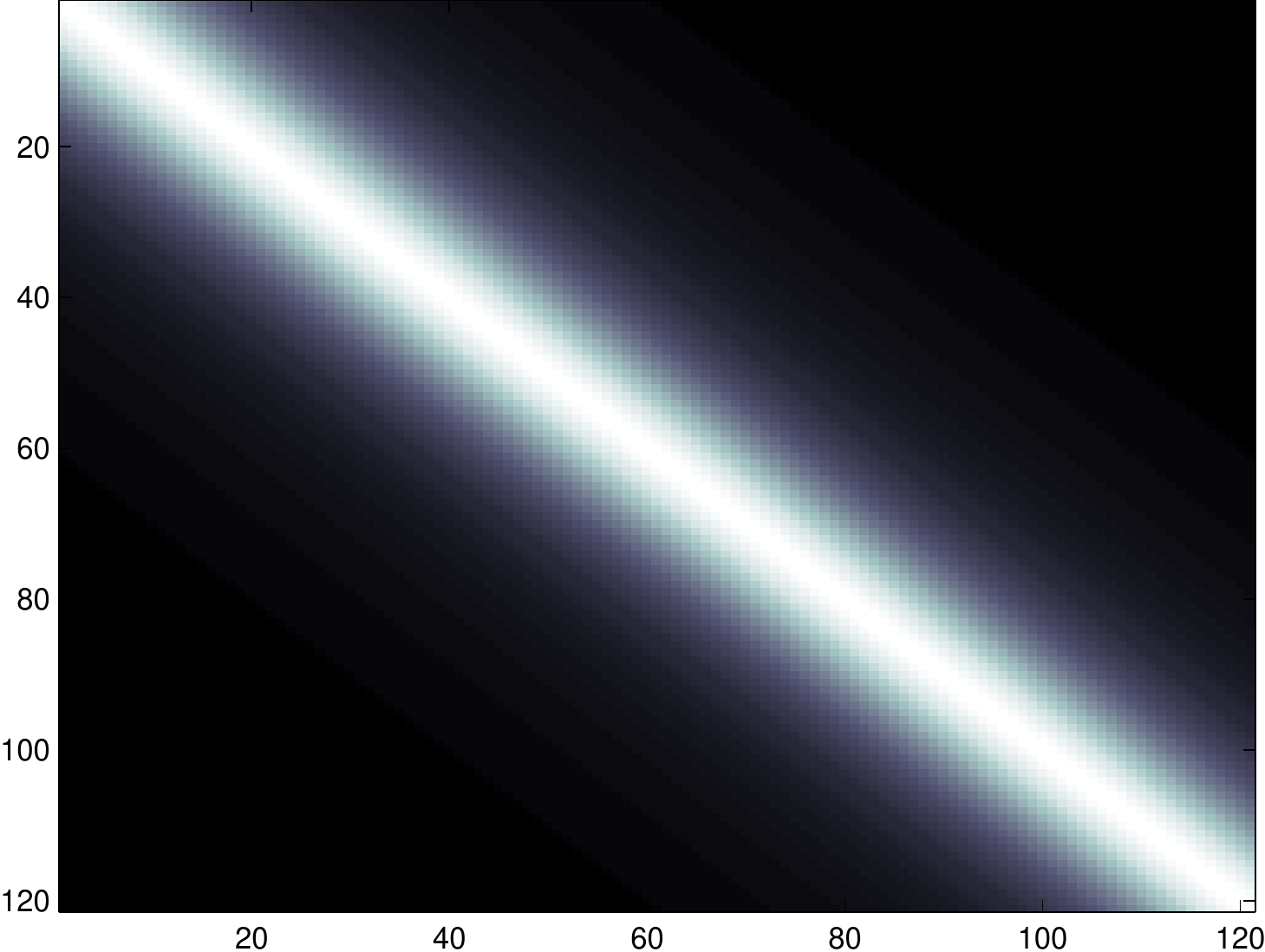}
	\caption{
	{\em Visualization of the sparsity of the discretization of the operators $S^{\mathrm{conv}}$ and $S^{\mathrm{t}}$ 
with $G = 100$ and $A/G = 50$.}	
	{\em Left:} $S^{\mathrm{t}} \in \Rr^{2N \times N}$ with $N = 120;$ gray level $128$ indicates that the corresponding matrix entry equals $0,$ black and white indicate negative and positive entries, respectively.
	{\em Right:} $S^{\mathrm{conv}} \in \Rr^{N \times N}$ with $N = 120$; white indicates a large positive value of the matrix entry whereas dark black indicates a matrix entry equal to $0$ here.
}		
		 \label{fig:Stime}
\end{figure}

\setlength{\figureheight}{4.5cm} \setlength{\figurewidth}{0.42\textwidth}

\begin{figure}
	\hspace{-7mm} \includegraphics[width=.55\textwidth]{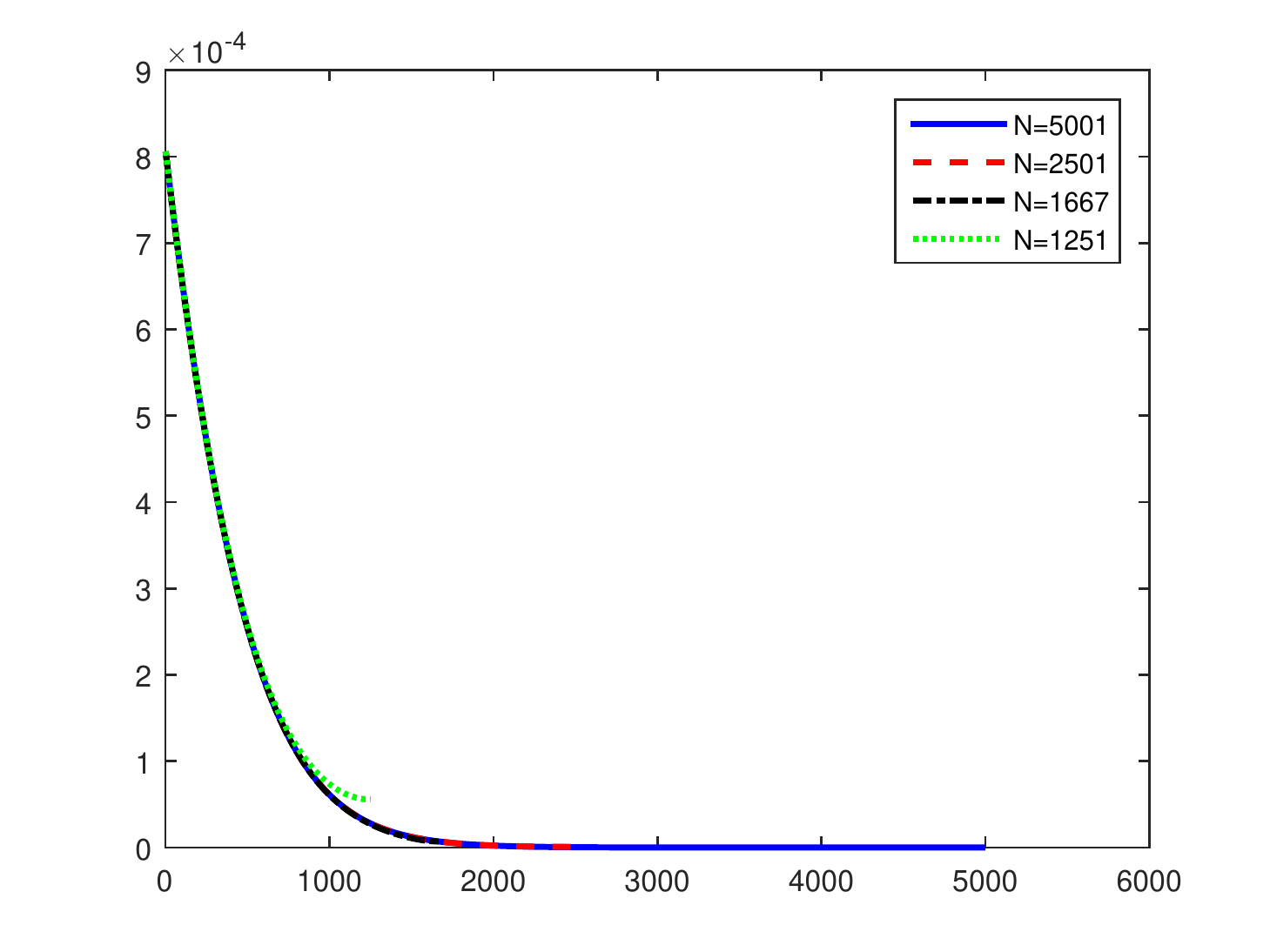} 
	\hspace{-6mm} \includegraphics[width=.55\textwidth]{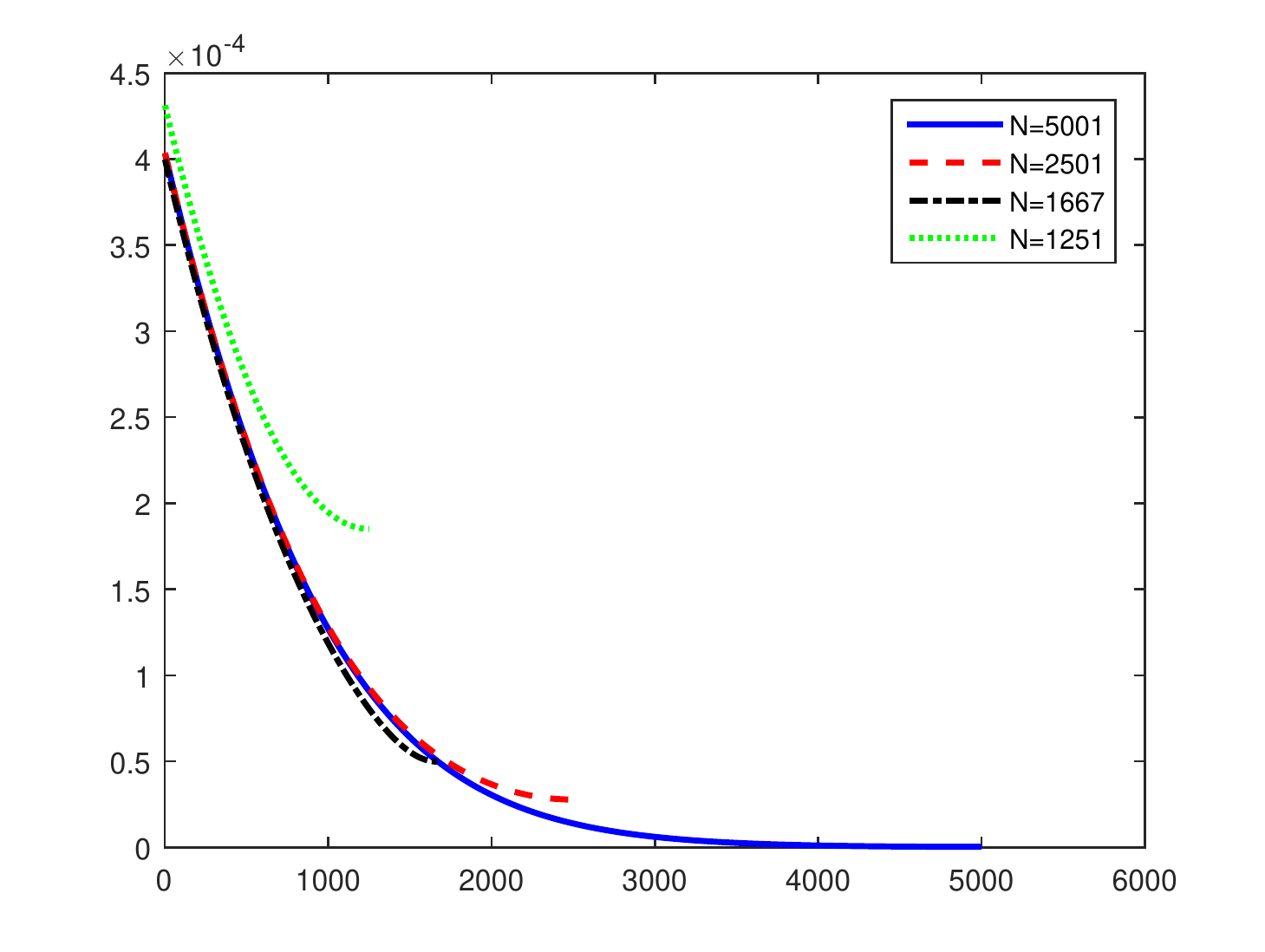}
	\caption{{\em Singular values of discretizations of $S^{\mathrm{conv}}.$}
		{\em Left:} For a field of view width $A/G=125,$ we discretized $S^{\mathrm{conv}}$ using $N$ sample points.
		We display the corresponding singular values in decreasing order.
		{\em Right:} For a field of view width $A/G=250,$ we discretized $S^{\mathrm{conv}}$ using $N$ sample points.
		We display the corresponding singular values in decreasing order as well.
		For a smaller ratio $A/G,$ we observe a good approximation already for rougher discretizations.}
	\label{fig:ApproxSingSconv}	
\end{figure}

\begin{figure}
	\hspace{-5mm}\includegraphics[width=.55\textwidth]{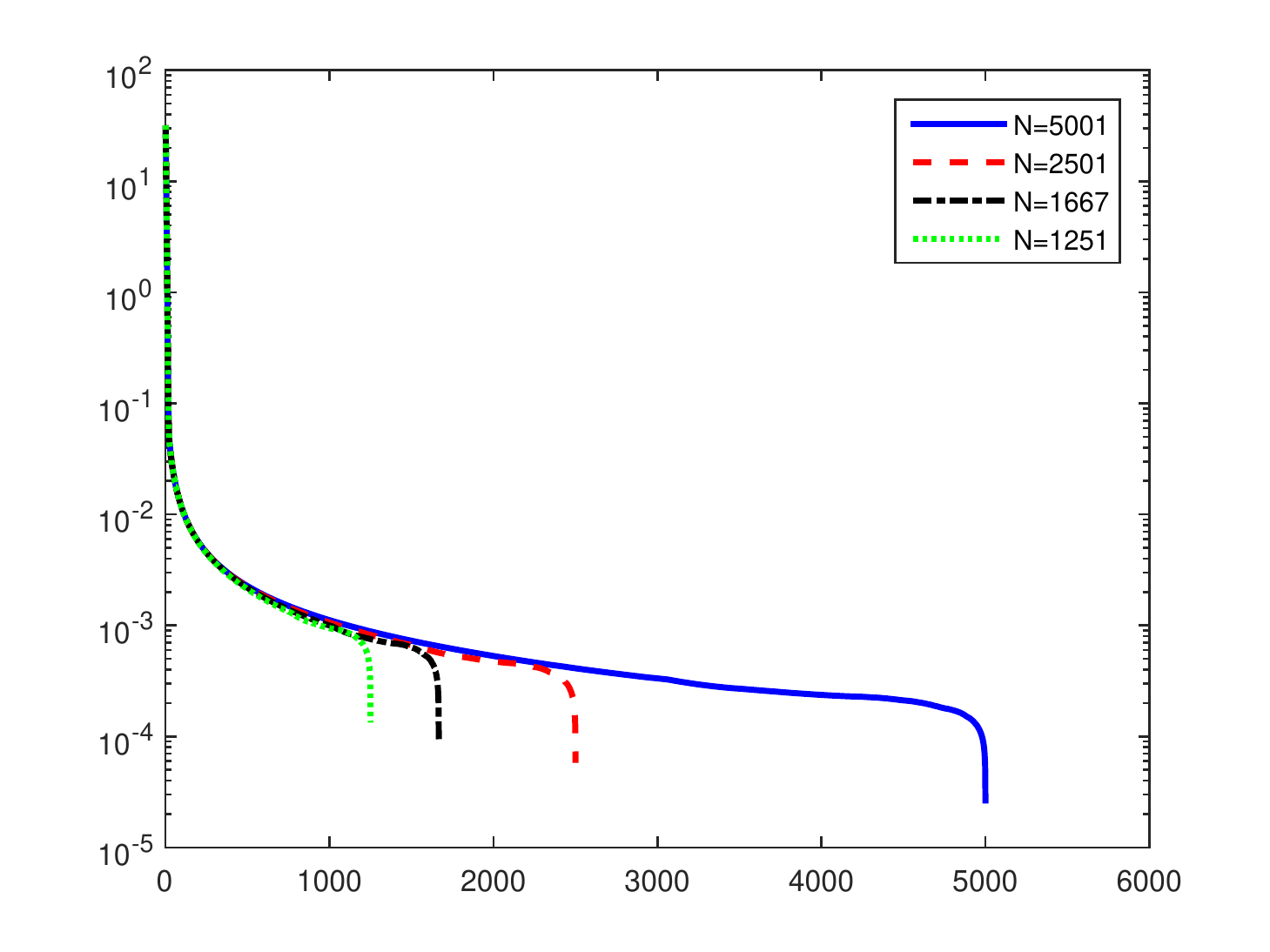}
	\hspace{-6mm}\includegraphics[width=.55\textwidth]{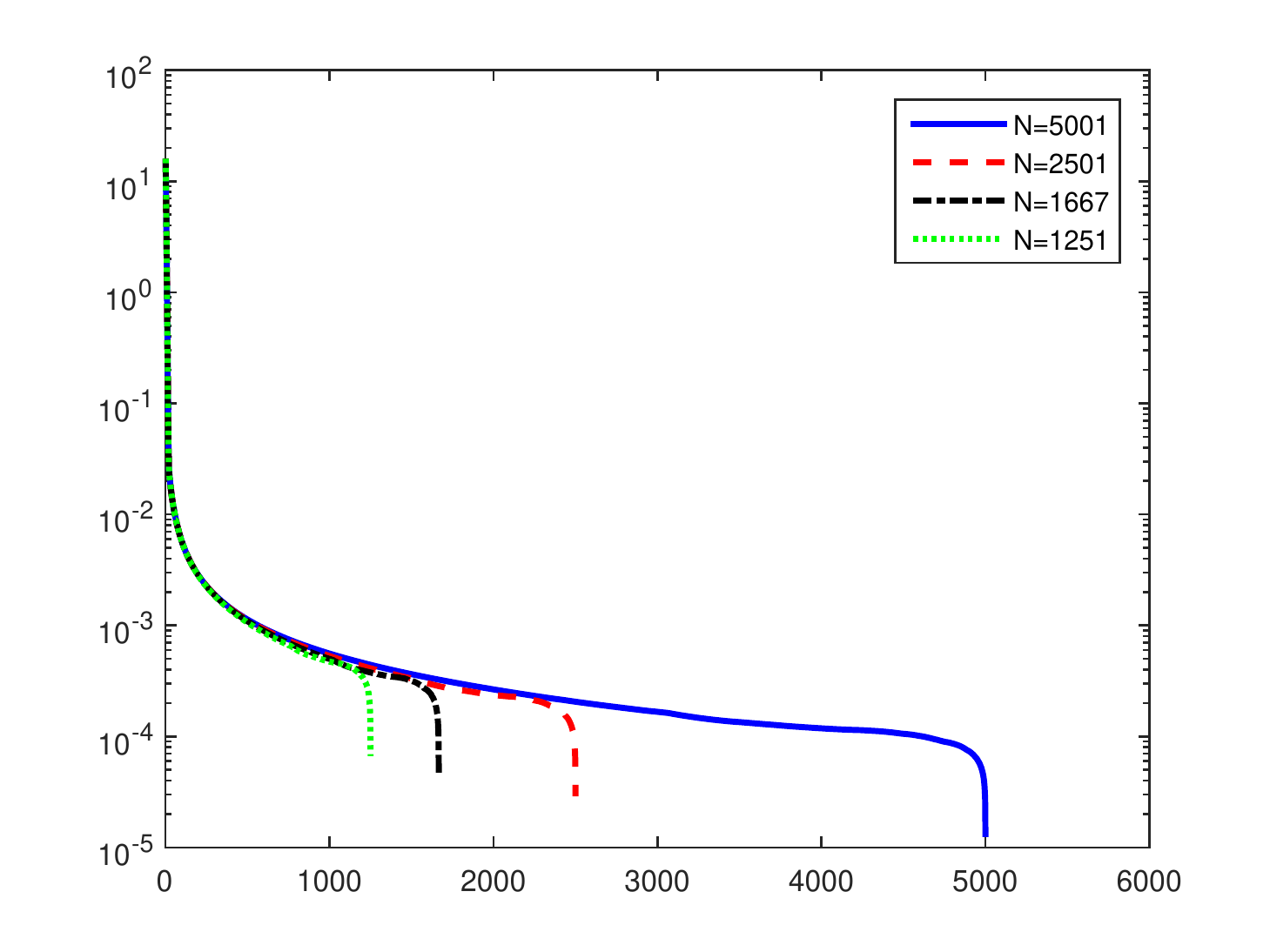}
	\caption{
		{\em Singular values of discretizations of $S^{\mathrm{t}}.$}
		{\em Left:} For a field of view width $A/G=125,$ we discretized $S^{\mathrm{t}}$ using $N$ sample points.
		We display the corresponding singular values in decreasing order.
		{\em Right:} For a field of view width $A/G=250,$ we discretized $S^{\mathrm{t}}$ using $N$ sample points.
		We display the corresponding singular values in decreasing order as well.
		We observe that the singular values initially decrease extremly fast. 	
	    (Therefore, we use a logarithmic scale for $y$-axis here.)
		Further, we observe that there are some very small singular values for each discretization. 		
		}	
	\label{fig:ApproxSingSt}	
\end{figure}

\begin{figure}
	\setlength{\figureheight}{5cm} \setlength{\figurewidth}{0.75\textwidth}
	\includegraphics[width=1\textwidth]{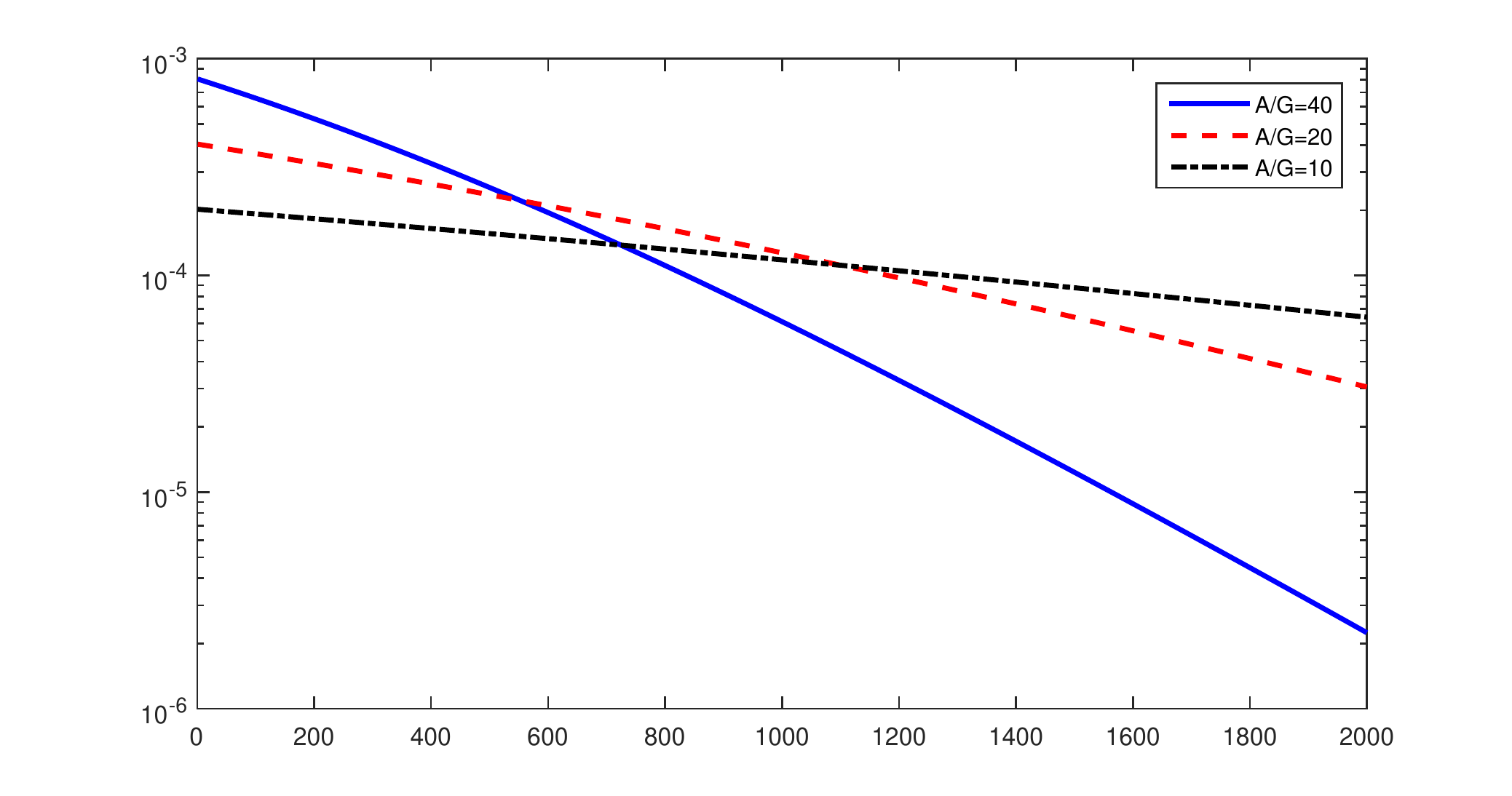}
	\caption{{\em Decay of the singular values of $S^{\mathrm{conv}}$.} The first 2000 singular values of $S^{\mathrm{conv}}$ discretized using $N=5001$ sampling points 
         for different
		 field of view widths $A/G.$
		 We view these singular values as a good approximation of the singular values of $S^{\mathrm{conv}}.$
		 From the semi-logarithmic plot we observe a decay of the singular values that
		 complements our analytic findings on expontial decay in Proposition \ref{prop:exponentialdecay}. 	
		 } \label{fig:DecaySingSconv}
\end{figure}

\begin{figure}
	\setlength{\figureheight}{5cm} \setlength{\figurewidth}{0.75\textwidth}
	\includegraphics[width=1\textwidth]{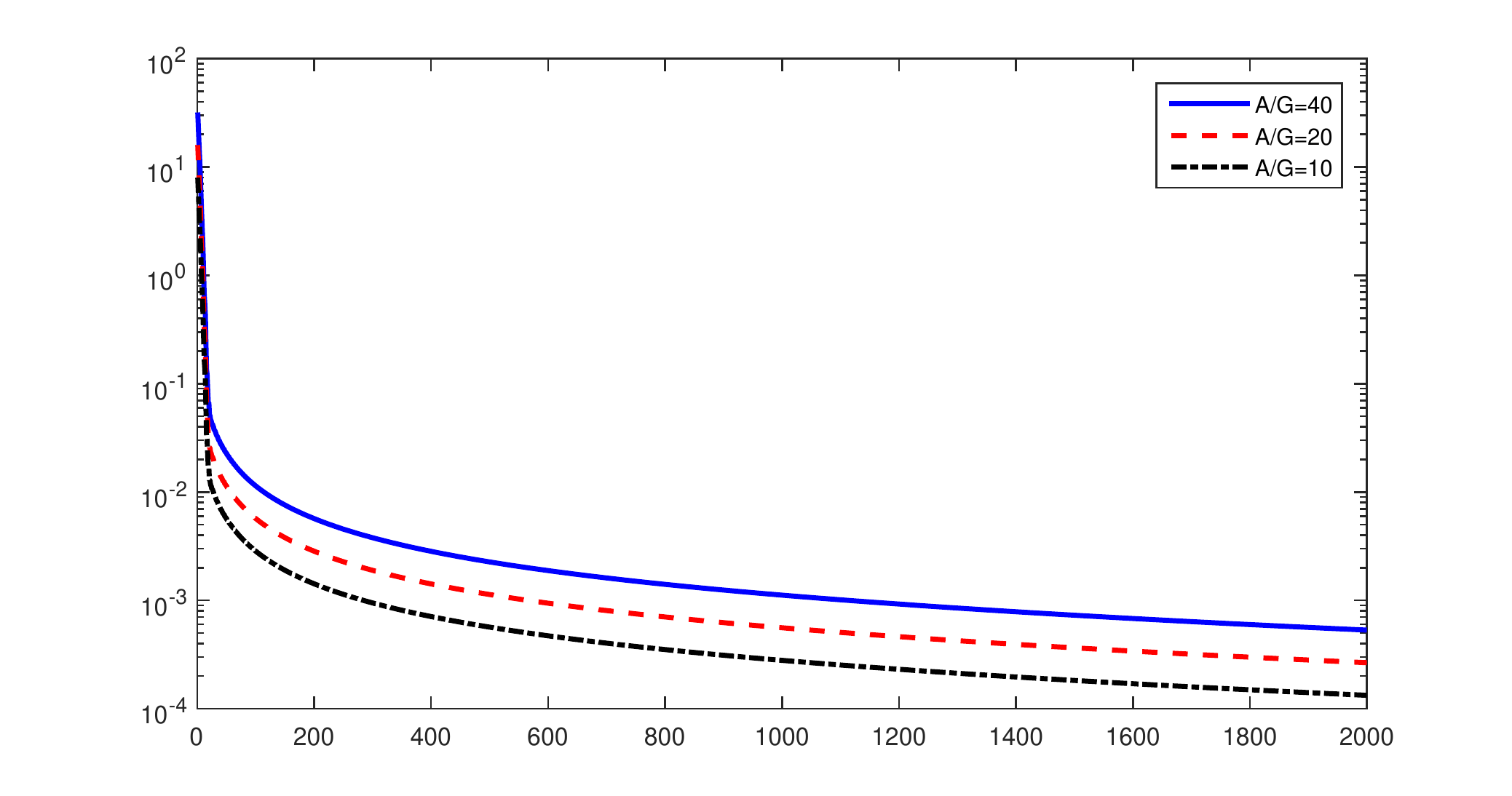}
	\caption{
		The first 2000 singular values of $S^{\mathrm{t}}$ discretized using $N=5001$ sampling points for different
		field of view widths $A/G.$
		We view these singular values as a good approximation of the singular values of $S^{\mathrm{t}}.$
		The semi-logarithmic plot shows that the singular values decrease very fast initially but then have 
a slower decay than that of the approximation of $S^{\mathrm{conv}}$ in Figure~\ref{fig:DecaySingSconv}.	
		}
	\label{fig:DecaySingSt}
\end{figure}

We now complement our analytic investigations by some numerical studies.
Here, we are particularly interested in the singular value decomposition of the MPI forward operator $S^{\mathrm{t}}$ and,
as its central part, of the core operator $S^{\mathrm{conv}}$ which describes the concatenation of the convolution operator and the restriction operator to the field of view.

Due to the concatenation of differently shaped operators in $S^{\mathrm{t}}$ and $S^{\mathrm{conv}}$, the explicit structure of the singular value decomposition --beyond our statement on the exponential decay of $S^{\mathrm{conv}}$ in 
Proposition \ref{prop:exponentialdecay}-- is analytically hard to access.

In the following we are going to access the singular values of $S^{\mathrm{conv}}$ and $S_1^{\mathrm{t}}$ numerically. 
To this end we have to discretize the corresponding operators. To discretize the field of view $[-A/G, A/G]$, we use a grid of $N$
equidistant points on which function values are sampled. The operator $S^{\mathrm{conv}}$ is then implemented as a discrete $N \times N$ convolution 
matrix acting on vectors of length $N$. A corresponding realization for $N = 120$ is illustrated in Figure~\ref{fig:Stime}. 
Also for the time domain an equidistant grid is used. To illustrate the entire period of the signal $u$, we use $[0,T]$
as underlying time domain and sample this interval on an equidistant larger grid with $K N$, $K \in \Nn$, points. Since
the space variable $y$ is related to the time variable by $y = \frac{A}{G} \cos \omega_0 t$, we regrid the rows of the discretized
matrix $S^{\mathrm{conv}}$ using linear interpolation. Then, the discretized time operator $S_1^{\mathrm{t}}$ is obtained by
multiplying the values of the function $\gamma_{1,A}'$ at the time sampling points. A realization of $S_1^{\mathrm{t}}$ for $N = 120$ and
$K = 2$ is shown in Figure~\ref{fig:Stime}.

We recall that, with increasing the resolution $N$ of the discretization, the singular values of the discretized operators
converge toward the singular value of the integral operators (which, in our case, are compact operators.)
More precisely, we order the eigenvalues of the integral operator and its discretization by magnitude and denote them by $\sigma_m$
for the integral operator and by $\sigma_m^{(l)}$ for the $l$th level discretization, respectively.
Then,
\begin{align}
	\sigma_m^{(l)}   \to  \sigma_m    \qquad \text { as } \qquad l \to \infty.
\end{align}

In Figure~\ref{fig:ApproxSingSconv} we illustrate this convergence for the particular instance of the operator $S^{\mathrm{conv}}.$
We display discretizations of $S^{\mathrm{conv}}$ using $N=1251,1667,$ $2501,5001$ sample points, gradient strength $G = 5000$ and 
two different sizes for the
field of view $A/G=125$ and $A/G=250.$ We observe faster approximation for smaller ratio $A/G.$	

In Figure~\ref{fig:ApproxSingSt} we perform an analogous experiment for the MPI operator $S^{\mathrm{t}}:$
we illustrate the approximation of $S^{\mathrm{t}}$ using $N=1251,1667,2501,5001$ sample points, $G = 5000$ and two different interval widths
$A/G=125$ and $A/G=250$ as above. For the time discretization we used $4N$ sample points for each instance of $N,$ respectively. 
We observe that the first singular values are decreasing extremely fast. Further, we observe that 
there are some very small singular values for each discretization.
Further experiments (we omit at this place) suggest that the number of these small singular values in the discretization increases when the number of time discretization points is reduced.

We observed that, for the presented examples, the first $2000$ singular values of the discretizations using 
$N=5001$ sample points rather well-approximate the first $2000$ singular values of $S^{\mathrm{conv}}$ and $S^{\mathrm{t}},$ respectively.
Therefore, in Figure~\ref{fig:DecaySingSconv},
we display the largest 2000 singular values of $S^{\mathrm{conv}}$ using $N=5001$ sampling points for discretization.
We use $G = 5000$ and the ratios $A/G =40; A/G =20; A/G = 10$ for the field of view.
Viewing these singular values as an approximation of the singular values of $S^{\mathrm{conv}},$
the semi-logarithmic plot displays a rapid initial decay of the singular values
complementing our analytic findings on expontial decay.	

In Figure~\ref{fig:DecaySingSt}, we display
	the first 2000 singular values of $S^{\mathrm{t}}$ discretized using $N=5001$ sampling points for different
		field of view lengths $A/G=40; A/G =20; A/G = 10$ and gradient strength $G = 5000$. 
		We view them as a good approximation of the largest 2000 singular values of $S^{\mathrm{t}}.$
		We see that the singular values decrease very fast initially.
	 However, comparing with Figure~\ref{fig:DecaySingSconv} we observe a slower decay than that of the approximation of $S^{\mathrm{conv}}$.

\section{Conclusion and Discussion}\label{sec:Conclusion}

In this work, we have analyzed the continuous 1D MPI imaging operator from a mathematical point of view.
We have derived a mathematical setup for the MPI reconstruction problem based on a linear operator equation between Hilbert spaces.
We have factorized the MPI imaging operator into basic building blocks and analyzed these building blocks systematically. 
In turn, we used this decomposition to analyse the forward 1D MPI operator. 
We were able to conclude that the reconstruction problem in MPI is severely ill-posed in the sense of inverse problems. In particular, as a major source for the ill-posedness of the 
problem we identified the convolution structure inside the imaging kernel. Restricted to the field of view this ill-posedness is reflected in an exponential decay of the 
singular values of the composed convolution-restriction operator. The exponential decay  herein can be attributed to the analycity 
of the Langevin function modelling the 1D particle magnetization. For the standard cosine trajectory used in 1D MPI, a second source of ill-posedness is given by 
the multiplication with a vanishing velocity at the boundaries of the field of view. The asymptotic exponential decay of the singular values of the core convolution operator was derived 
analytically, whereas final numerical experiments suggest a fast decay of the singular values also for the entire composed MPI imaging operator. 

In two and three spatial dimensions, the MPI forward operator can be decomposed similarly, but the single building blocks are more involved. For instance, in the multivariate setup, the MPI core operator consists of a matrix-valued convolution operator compared with an orinary convolution operator in 1D; cf. \cite{marz2016model} for an analysis of the MPI core operator in two and three spatial dimensions. Another difference arises from the involved signal generating trajectories. 
In a higher dimensional setup, the data are measured along a one-dimensional curve within a higher-dimensional field of view is covered. Hence, the scan along a trajectory only yields partial information whereas in a univariate situation the trajectory covers 
the whole univariate field of view. 
Suming up, the findings obtained in this article do not directly carry over to the higher-dimensional setup but, however, they may serve as a guide how to deal with their counterparts in the more involved multivariate setup. 

\section*{Acknowledgements}
All authors of this article were involved in the activities of the DFG-funded scientific network MathMPI
and thank the German Research Foundation for the support (ER777/1-1).
Wolfgang Erb and Andreas Weinmann thank Karlheinz Gr\"ochenig 
for a valuable discussion at the Dolomites Workshop on Approximation 2017 and the Rete Italiana di Approssimazione (RITA) for supporting a research visit in Padova.
Martin Storath acknowledges support by the German Research Foundation DFG (STO1126/2-1).
Andreas Weinmann acknowledges support by the German Research Foundation DFG (WE5886/4-1, WE5886/3-1).

\end{document}